%% file: main.tex
\documentclass[conference]{IEEEtran}
\usepackage{amsmath,amssymb,amsfonts}
\usepackage{algorithmic}
\usepackage{comment}
\usepackage{graphicx}
\usepackage{textcomp}
\usepackage{xcolor}
\usepackage{enumitem}
\usepackage{amsthm}
\usepackage{mathtools}
\usepackage{booktabs}
\usepackage[colorlinks]{hyperref}
\hypersetup{colorlinks, breaklinks, citecolor=teal,filecolor=black}
\usepackage[nosort,capitalise]{cleveref}
\crefname{enumi}{}{}
\crefname{equation}{}{}

\usepackage{tikz}
\usetikzlibrary{matrix}
\input{notations}

\input{bibliography_setup}
\newtheorem{theorem}{Theorem}
\newtheorem{conjecture}{Conjecture}

\newtheorem{definition}{Definition}
\newtheorem{lemma}[theorem]{Lemma}

\newtheorem{remark}[theorem]{Remark}

\def\BibTeX{{\rm B\kern-.05em{\sc i\kern-.025em b}\kern-.08em
    T\kern-.1667em\lower.7ex\hbox{E}\kern-.125emX}}
\begin{document}

\title{The Lovász number of random circulant graphs
}

\author{
\IEEEauthorblockN{Afonso S. Bandeira}
\IEEEauthorblockA{\textit{ETH Zurich}}
\and
\IEEEauthorblockN{Jarosław Błasiok}
\IEEEauthorblockA{\textit{ETH Zurich}}
\and
\IEEEauthorblockN{Daniil Dmitriev}
\IEEEauthorblockA{\textit{ETH Zurich}}
\and
\IEEEauthorblockN{Ulysse Faure}
\IEEEauthorblockA{\textit{ETH Zurich}}
\and
\IEEEauthorblockN{Anastasia Kireeva}
\IEEEauthorblockA{\textit{ETH Zurich}}
\and
\IEEEauthorblockN{Dmitriy Kunisky}
\IEEEauthorblockA{\textit{Johns Hopkins University}}
}

\maketitle

\begin{abstract}
This paper addresses the behavior of the Lovász number for dense random circulant graphs.
The Lovász number is a well-known semidefinite programming upper bound on the independence number.
Circulant graphs, an example of a Cayley graph, are highly structured vertex-transitive graphs on integers modulo n, where the connectivity of pairs of vertices depends only on the difference between their labels.
While for random circulant graphs the asymptotics of fundamental quantities such as the clique and the chromatic number are well-understood, characterizing the exact behavior of the Lovász number remains open.
In this work, we provide upper and lower bounds on the expected value of the Lovász number and show that it scales as the square root of the number of vertices, up to a log log factor.
Our proof relies on a reduction of the semidefinite program formulation of the Lovász number to a linear program with random objective and constraints via diagonalization of the adjacency matrix of a circulant graph by the discrete Fourier transform (DFT).
This leads to a problem about controlling the norms of vectors with sparse Fourier coefficients, which we study using results on the restricted isometry property of subsampled DFT matrices.
\end{abstract}


\input{content/introduction}
\input{content/preliminaries}
\input{content/lp}
\input{content/discussion}

\input{content/technical_lemmas}
\section*{Acknowledgment}
We thank Filip Kovačević for insightful discussions in the beginning of the project. DD is partly supported by the ETH AI Center and the ETH Foundations of Data Science initiative.
\printbibliography

\end{document}

%% file: notations.tex
\DeclareMathOperator{\Tr}{Tr}

\DeclareMathOperator{\E}{\mathbf{E}}
\DeclareMathOperator{\Pro}{\mathbf{P}}

\newcommand{\ltn}{\vartheta}

\DeclareFontFamily{U}{mathx}{}
\DeclareFontShape{U}{mathx}{m}{n}{<-> mathx10}{}
\DeclareSymbolFont{mathx}{U}{mathx}{m}{n}
\DeclareMathAccent{\wh}{0}{mathx}{"70}
\DeclareMathAccent{\wc}{0}{mathx}{"71}
\DeclareMathAccent{\wt}{0}{mathx}{"72}

\ifdefined\C
\renewcommand{\C}{\mathbf{C}}
\else
\newcommand{\C}{\mathbf{C}}
\fi

\newcommand{\R}{\mathbf{R}}
\newcommand{\N}{\mathbf{N}}
\newcommand{\Z}{\mathbf{Z}}

\newcommand{\cO}{\mathcal{O}}
\newcommand{\co}{{\scriptstyle\mathcal{O}}}

\DeclarePairedDelimiter{\braket}{\langle}{\rangle}%
\DeclarePairedDelimiter{\abs}{\lvert}{\rvert}%
\DeclarePairedDelimiter{\norm}{\lVert}{\rVert}%

\DeclarePairedDelimiterXPP{\Exp}[1]{\E}[]{}{#1}
\DeclarePairedDelimiterXPP{\Prob}[1]{\Pro}(){}{#1}
\DeclarePairedDelimiterX{\tuple}[1](){#1}
\DeclarePairedDelimiterX{\set}[1]\{\}{#1}
\DeclarePairedDelimiterXPP{\landauO}[1]{\cO}(){}{#1}
\DeclarePairedDelimiterXPP{\landauo}[1]{\co}(){}{#1}
\DeclarePairedDelimiterXPP{\landauOprec}[1]{\cO_\prec}(){}{#1}

\newcommand{\e}{\varepsilon}

\newif\ifnotes
 \notestrue

%% file: bibliography_setup.tex
\usepackage[giveninits=true,url=false,doi=false,isbn=false,eprint=true,datamodel=mrnumber,sorting=nty,maxcitenames=4,maxbibnames=99,backref=false,block=space,backend=biber,bibstyle=phys]{biblatex} 
\AtEveryBibitem{\clearfield{month}}
\AtEveryCitekey{\clearfield{month}} 
\renewbibmacro{in:}{}
\DeclareFieldFormat[article]{title}{\emph{#1}} 
\DeclareFieldFormat{mrnumber}{\ifhyperref{\href{http://www.ams.org/mathscinet-getitem?mr=#1}{\nolinkurl{MR#1}}}{\nolinkurl{#1}}}
\DeclareFieldFormat{pmid}{\ifhyperref{\href{https://www.ncbi.nlm.nih.gov/pubmed/#1}{\nolinkurl{PMID#1}}}{\nolinkurl{#1}}}
\DeclareFieldFormat{eprint}{\ifhyperref{\href{https://arxiv.org/abs/#1}{\nolinkurl{arXiv:#1}}}{\nolinkurl{#1}}}
\renewbibmacro*{doi+eprint+url}{%
  \iftoggle{bbx:doi}{\printfield{doi}}{}
  \newunit\newblock%
  \printfield{mrnumber}%
  \newunit\newblock%
  \printfield{pmid}%
  \newunit\newblock%
  \printfield{eprint}%
  \iftoggle{bbx:url}{\usebibmacro{url+urldate}}{}}

%% file: content/introduction.tex
\section{Introduction}
The Lovász number $\vartheta$ is a well-known statistic of an arbitrary simple undirected graph $G$. 
As Lovász first observed in \cite{lovasz1979shannon}, one can define a number $\ltn(G)$ 
as the value of a certain semidefinite program (SDP) 
whose constraints depend on the adjacency matrix of $G$.
The Lovász number provides an upper bound on the Shannon capacity of the graph and satisfies the following inequalities:
\begin{equation}
\label{eq:theta_ineq}
\omega(G) \leq \ltn(\overline{G}) \leq \chi(G),
\end{equation}
where $\omega(G)$ is the size of the largest clique in $G$, $\chi(G)$ is the chromatic number of $G$, and $\overline{G}$ is the complement of $G$. 
This observation is remarkable, since $\vartheta$ is computable in polynomial time, while $\omega$ and $\chi$ are famously NP-hard to compute. 

The Lovász number has been studied for a variety of random graph models including the classical Erd\H{o}s-R\'{e}nyi (ER) random graph $G(n, p)$. 
Its expected value was first studied by Juh{\'a}sz~\cite{juhasz1982asymptotic}, 
who showed that $\E \ltn(G) = \Theta(\sqrt{n/p})$ for $\frac{\log^6 n}{n} \le p \le 1/2$. 
For $p=1/2$, Arora and Bhaskara~\cite{aroranote} showed that \(\ltn(G)\) concentrates around its median in an interval of polylogarithmic length.
In the sparse regime $p < n^{-1/2}$, it has been further shown that \(\ltn(G)\) concentrates around its median in an interval of constant length~\cite{coja2005lovasz}.
To the best of our knowledge, determining the correct constant in the $\Theta(\sqrt{n/p})$ asymptotic remains an open question. 

\begin{figure}[ht]
\centering
\begin{tabular}{cc}
\begin{tikzpicture}[scale=1.2]
  \foreach \i in {0,...,8} {
    \node[draw, circle, inner sep=2pt] (v\i) at ({360/9 * \i}:1.5cm) {};
  }
  \foreach \i in {0,...,8} {
    \pgfmathtruncatemacro{\nextTwo}{mod(\i+2,9)}
    \draw[blue, thick] (v\i) -- (v\nextTwo);
    \pgfmathtruncatemacro{\nextThree}{mod(\i+3,9)}
    \draw[red, thick] (v\i) -- (v\nextThree);
  }
\end{tikzpicture}
&

\begin{tikzpicture}[every node/.style={minimum size=0.3cm, anchor=center, font=\footnotesize}, scale=0.6]
\matrix (m3) [matrix of nodes,
    left delimiter={(},
    right delimiter={)},
    row sep=-\pgflinewidth, column sep=-\pgflinewidth
]{
  $\cdot$ & $\cdot$ & $\textcolor{blue}{\mathbf{1}}$ & $\textcolor{red}{\mathbf{1}}$ & $\cdot$ & $\cdot$ & $\textcolor{red}{\mathbf{1}}$ & $\textcolor{blue}{\mathbf{1}}$ & $\cdot$ \\
  $\cdot$ & $\cdot$ & $\cdot$ & $\textcolor{blue}{\mathbf{1}}$ & $\textcolor{red}{\mathbf{1}}$ & $\cdot$ & $\cdot$ & $\textcolor{red}{\mathbf{1}}$ & $\textcolor{blue}{\mathbf{1}}$ \\
  $\textcolor{blue}{\mathbf{1}}$ & $\cdot$ & $\cdot$ & $\cdot$ & $\textcolor{blue}{\mathbf{1}}$ & $\textcolor{red}{\mathbf{1}}$ & $\cdot$ & $\cdot$ & $\textcolor{red}{\mathbf{1}}$ \\
  $\textcolor{red}{\mathbf{1}}$ & $\textcolor{blue}{\mathbf{1}}$ & $\cdot$ & $\cdot$ & $\cdot$ & $\textcolor{blue}{\mathbf{1}}$ & $\textcolor{red}{\mathbf{1}}$ & $\cdot$ & $\cdot$ \\
  $\cdot$ & $\textcolor{red}{\mathbf{1}}$ & $\textcolor{blue}{\mathbf{1}}$ & $\cdot$ & $\cdot$ & $\cdot$ & $\textcolor{blue}{\mathbf{1}}$ & $\textcolor{red}{\mathbf{1}}$ & $\cdot$ \\
  $\cdot$ & $\cdot$ & $\textcolor{red}{\mathbf{1}}$ & $\textcolor{blue}{\mathbf{1}}$ & $\cdot$ & $\cdot$ & $\cdot$ & $\textcolor{blue}{\mathbf{1}}$ & $\textcolor{red}{\mathbf{1}}$ \\
  $\textcolor{red}{\mathbf{1}}$ & $\cdot$ & $\cdot$ & $\textcolor{red}{\mathbf{1}}$ & $\textcolor{blue}{\mathbf{1}}$ & $\cdot$ & $\cdot$ & $\cdot$ & $\textcolor{blue}{\mathbf{1}}$ \\
  $\textcolor{blue}{\mathbf{1}}$ & $\textcolor{red}{\mathbf{1}}$ & $\cdot$ & $\cdot$ & $\textcolor{red}{\mathbf{1}}$ & $\textcolor{blue}{\mathbf{1}}$ & $\cdot$ & $\cdot$ & $\cdot$ \\
  $\cdot$ & $\textcolor{blue}{\mathbf{1}}$ & $\textcolor{red}{\mathbf{1}}$ & $\cdot$ & $\cdot$ & $\textcolor{red}{\mathbf{1}}$ & $\textcolor{blue}{\mathbf{1}}$ & $\cdot$ & $\cdot$ \\
};
\end{tikzpicture}
\\[0.0em]
\end{tabular}
\caption{Circulant graph on \(9\) vertices and its adjacency matrix (0's replaced by dots). Each vertex \(i\)  is connected to vertices \(\textcolor{blue}{i + 2, i+7}, \textcolor{red}{i + 3}\), and \(\textcolor{red}{i + 6} \mod 9\).}
\label{fig:circ_graphs}
\end{figure}
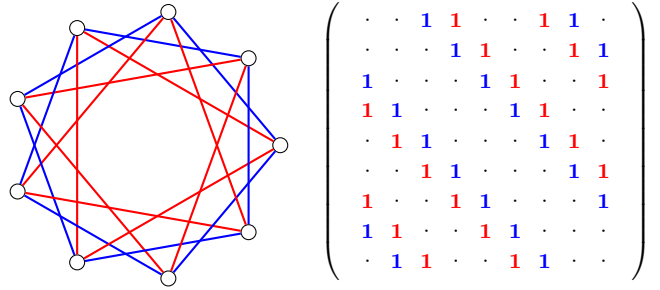
In this work, we focus on a class of random \emph{circulant} graphs (RCGs), a family of vertex-transitive graphs with a circulant adjacency matrix; see~\Cref{fig:circ_graphs} and~\Cref{def:circ_graph,def:rand_circ}. 
We emphasize that RCGs are fully determined by the connectivity of any given single vertex. Therefore, a dense RCG can be generated with \(\frac{n - 1}{2}\) random bits, where each bit affects the presence of \(n\) edges, in contrast to the \(\frac{n(n - 1)}{2}\) random bits in $G(n,1/2)$, each affecting just one edge.
In this sense, RCGs may be viewed as a ``partial derandomization'' of ER graphs.
Indeed, circulant graphs are precisely Cayley graphs on the group \(\mathbb{Z}_n\), and general random Cayley graphs have long been studied for similar purposes in theoretical computer science.

It is therefore of interest to understand to what extent the above results for ER graphs also apply to RCGs.
For dense RCGs, the asymptotics of the clique number and the chromatic number are well-understood:~\cite{green2005counting} showed a high-probability upper bound on the clique number \(\omega(G) = O(\log n)\), and later~\cite{green2016counting} proved that the chromatic number is at most $(1+o(1)) \frac{n} { 2 \log_2 n}$ with high probability. 
These results imply bounds on the Lovász number through~\Cref{eq:theta_ineq}, but the resulting upper and lower bounds are far apart.

Our main result is much sharper upper and lower bounds on the expected Lovász number of a dense RCG. 
\begin{theorem}
\label{thm:main}
    There exists a constant $C > 0$ such that, for a dense random circulant graph \(G\) on $n$ vertices (\Cref{def:rand_circ}),
    \begin{equation}
        \sqrt{n} \leq \E \ltn(G) \leq C \sqrt{n \log \log n}.
    \end{equation}
\end{theorem}

\noindent
\emph{Proof Strategy:}
Our proof of the upper bound in~\Cref{thm:main} relies on the algebraic structure of circulant graphs. First, following~\cite{magsino2019linear}, we transform the SDP formulation of \(\ltn(G)\) to a linear program (LP) using the fact that the circulant matrices are diagonalizable by a discrete Fourier transform (DFT)
\Cref{lem:lp_with_g} gives the resulting LP:
\begin{equation}
\label{eq:first_step}
\begin{aligned}
    \ltn(G) &= \max_{(y_0, \dots, y_{n-1}) \in \R^{n}} \braket{y, g}, \\ \text{ subject to }
    &\begin{cases} y_k = y_{n - k} \text{ for } k = 1,\ldots, n - 1, \\
    \norm{y}_1 = 1,
    y \geq \mathbf{0}, \\
    \braket{y, f_k} = 0\text{ for all edges } (0, k).
    \end{cases}
\end{aligned}
\end{equation}
Here, \(f_k\) is the \(k\)-th row of the DFT matrix \(F\), and \(g \coloneqq Fb\) for \(b \in \{\pm 1\}^{n}\) with $b_0=1$ and \(b_k = 1\) if \((0, k)\) is not an edge, and \(-1\) otherwise, for $1 \le k \le n-1$. We denote \(\mathbf{0} \coloneqq (0, \ldots, 0)\) and \(y \geq \mathbf{0}\) stands for entrywise positivity of \(y\).

The last constraint in~\Cref{eq:first_step} requires the Fourier transform of \(y\) to have a specific sparsity pattern. Uncertainty principles for the Fourier transform (see, e.g.,~\cite{bandeira2018discrete}) then suggest that all feasible vectors \(y\) must be dense \cite{demanet2014scaling}.
A quantitative version of this ``density'' would be enough to bound the LP.
To illustrate, suppose that $y$ is a feasible vector with $\|y\|_1 = 1$ and its mass is spread almost uniformly among its coordinates, i.e., that
\(\norm{y}_2 \leq \frac{c}{\sqrt{n}} \norm{y}_1 = \frac{c}{\sqrt{n}}\), for some constant \(c > 0\). Since \(\norm{g}_2 = n\), Cauchy-Schwarz inequality would give \(\braket{y, g} \leq \norm{y}_2 \norm{g}_2 \leq c\sqrt{n},\) proving upper bound in~\Cref{thm:main} without the extra \(\sqrt{\log \log n}\) factor.

The second part of our proof,~\Cref{lem:rip}, makes the aforementioned intuition rigorous, relying on the \textit{restricted isometry property} (RIP,~\Cref{def:rip}).
The $f_k$ in our constraints form a so-called \emph{subsampled DFT basis}, which is a random subset of the Fourier basis.
The RIP for such bases is in fact a celebrated topic in the compressed sensing literature. RIP was first introduced and studied for subsampled DFT bases in seminal work of Candès and Tao~\cite{candes2006near},
and since then, one of the central 
questions for compressed sensing is 
the number of $f_k$ needed for RIP to hold. 
\Cref{lem:rip_dft} describes a simplified version of the current best bound due to~\cite{haviv2017restricted} which is sufficient for our purposes.
Interestingly, our upper bound proof only uses the fact that feasible solutions of~\Cref{eq:first_step} lie on a (random) nullspace of a subsampled DFT matrix, and omits the positivity constraint \(y \geq \mathbf{0}\). However, as we discuss in~\Cref{sec:discussion}, we believe that this constraint is important for tighter results.

%% file: content/preliminaries.tex
\section{Preliminaries}
\paragraph*{Notation}
For \(n \in \N\), let \([n] \coloneqq \{0, \ldots, n - 1\}\). We index vectors and matrices by \([n]\): for \(x 
\in \R^n, x = (x_0, \ldots, x_{n-1})\). We write \(x \geq \mathbf{0}\) for entrywise positivity.
For \(n \in \N\), we denote by \(G = (V, E)\) a graph with vertex set \(V = [n]\) 
and edge set \(E \subseteq (V \times V) \setminus \{(k, k) \text{ for } k \in V\}\).
For a graph \(G = (V, E)\) we define its complement \(\overline{G} = (V, E')\),
where \(E' = \{(u, v) \text{ s.t. } u \neq v \text{ and } (u, v) \notin E\}\).
We use the standard asymptotic notation, $O(\cdot), \Omega(\cdot)$, and $\Theta(\cdot)$ to describe the order of the growth of functions associated with the limit of the graph dimension $n$. For \(x \in \R^n\), we denote \(\norm{x}_1 \coloneqq \sum_{k=0}^{n - 1} \abs{x_k}\), \(\norm{x}_2 \coloneqq \left(\sum_{k=0}^{n-1} x_k^2\right)^{1/2}\), and \(\norm{x}_{\infty} \coloneqq \max_{k}\abs{x_k}\).

\paragraph*{Discrete Fourier Transform}
Let \(F \in \C^{n \times n}\) be the discrete Fourier transform matrix: \(F_{jk} = \exp(-2 \pi i jk/n)\) for \(j, k \in [n]\). For \(k \in [n]\), let \(f_k\) denote the \(k\)-th row of \(F\).
We associate a matrix \(\wt F \in \R^{m \times n}\) to any RCG $G$ consisting of subsampled rows of \(F\).
\begin{definition}
\label{def:f_tilde}
For any RCG $G$, let $\wt F \equiv \wt F (G) \in \C^{m \times n}$ (with $m$ the number of neighbors of $0$ in $G$)  be defined as a submatrix of \(F\), including row \(f_k\) if $(0,k) \in E(G).$
\end{definition}

\begin{definition}
 The \emph{Lovász theta number} \(\ltn(G)\) is defined as the solution to the following SDP (\(J\) is the all-ones matrix),
\begin{equation}
\label{def:ltn}
\begin{aligned}
    \ltn(G) \coloneqq \max_{X \in \R^{n \times n}}
    \Big\{&\langle X, J \rangle \text{, such that } X \succeq 0, \Tr X = 1, \\
    &X_{ij} = 0\text{ for all }(i, j) \in E(G)\Big\}.
\end{aligned}
\end{equation}
\end{definition}
\begin{definition}
\label{def:circ_graph}
 A graph on $n$ vertices is called \emph{circulant} if there is an ordering of its vertices such that its adjacency matrix is circulant. Equivalently, a circulant graph is a Cayley graph of a cyclic group \(\Z_n\).
\end{definition}

This definition implies that a circulant graph is described by listing the neighbors of a single root vertex (say vertex $0$), since $(i,j)\in E \iff (0,i-j)\in E.$
In this text, we focus on \emph{dense random} circulant graphs.
    \begin{definition}
    \label{def:rand_circ}
    For odd \(n\), a \emph{dense random circulant graph (RCG)} is a random Cayley graph of a cyclic group \(\Z_n\). It is obtained in the following way: uniformly sample \(x \in \{0, 1\}^m, m={\frac{n - 1}{2}},\) and define the first row of the adjacency matrix as
    \begin{equation}
        R = (0\ \ x\ \ \overset{\leftarrow}{x}),
    \end{equation}
    where \(\overset{\leftarrow}{x}_i \coloneqq x_{m-i-1}\). Circulate $R$ to obtain the complete adjacency matrix.
    \end{definition}

For a circulant graph \(G\) we define a vector \(g \coloneqq Fb\), where \(b \in \{\pm 1\}^{n}\) with $b_0=1$ and \(b_k = 1\) if \((0, k)\) is not an edge, and \(-1\) otherwise, for $1 \le k \le n-1$.

\begin{definition}[Restricted isometry property]
\label{def:rip}
    A matrix \(A \in \C^{q \times n}\) is said to satisfy \((k, \varepsilon)\)-restricted isometry property, for \(k \leq n\) and \(\varepsilon \in (0, 1)\), if for all \(k\)-sparse \(x \in \C^{n}\) we have that 
    \begin{equation}
        (1 - \varepsilon) \norm{x}_2^2 \leq \norm{Ax}_2^2 \leq (1 + \varepsilon) \norm{x}_2^2. 
    \end{equation}
\end{definition}

%% file: content/lp.tex
\section{Proof of main theorem}

Let $G$ be a circulant graph. As noted in~\cite{magsino2019linear}, for circulant graphs the SDP formulation of the Lovász number can be rewritten as the following linear program:

\begin{equation}
\begin{aligned}
    \ltn(G) &= \max_{x \in \R^{n}} \sum_{i \in [n]} x_i, \\ \text{ subject to }
    &\begin{cases} x_k = x_{n - k} \text{ for all } k \in [n] \setminus \{0\}, \\
    x_0 = 1,
    Fx \geq \mathbf{0}, \\
    x_k = 0\text{ for all edges } (0, k),
    \end{cases}
\end{aligned}
\end{equation}

\begin{table}[ht]
\centering
\caption{Four equivalent LPs for \(\ltn(G)\).}
\label{table:4_lps}
\renewcommand{\arraystretch}{1.05}
\begin{tabular}{@{}p{0.47\linewidth} p{0.47\linewidth}@{}}
\multicolumn{2}{c}{\textbf{'time' domain}} \\[3pt]
\toprule
\textbf{Primal} & \textbf{Dual}\\
\midrule
\(
\begin{aligned}
\max_{x \in \R^n} \ &\sum_{i} x_i\\
\text{s.t. } &x_k = x_{n - k} \\
&\quad \text{for all } k \in [n] \setminus \{0\}, \\
&x_0 = 1,\ Fx \ge \mathbf{0} ,\\
&x_k=0 \\
&\quad \text{for all }(0, k) \in E(G).
\end{aligned}
\)
&
\(
\begin{aligned}
\min_{z \in \R^n} &1 + \sum_i z_i \\
\text{s.t. } &z_k = z_{n - k} \\
&\quad \text{for all } k \in [n] \setminus \{0\}, \\
&z \geq \mathbf{0}, \\
&\langle z, f_k \rangle = -1 \\
&\quad \text{for all } (0, k) \in E(\overline{G}).
\end{aligned}
\)
\\
\bottomrule
\\[-0.2em]
\multicolumn{2}{c}{\textbf{'frequency' domain}} \\[3pt]
\toprule
\textbf{Primal} & \textbf{Dual}\\
\midrule
\(
\begin{aligned}
\max_{y \in \R^n}\ &n y_0\\
\text{s.t. } &y_k = y_{n - k} \\
&\quad \text{for all } k \in [n] \setminus \{0\},  \\
&\norm{y}_1 = 1,\ y \geq \mathbf{0},\\
&\langle y, f_k \rangle = 0 \\
&\quad \text{for all } (0, k) \in E(G).
\end{aligned}
\)
&
\(
\begin{aligned}
\min_{t \in \R^n} &1 + n t_0\\
\text{s.t. } &t_k = t_{n - k} \\
&\quad \text{for all } k \in [n] \setminus \{0\},\\
&Ft \geq \mathbf{0}, \\
&t_k = -1 / n \\
&\quad \text{for all } (0, k) \in E(\overline{G}).
\end{aligned}
\)
\\
\bottomrule
\end{tabular}
\end{table}

\Cref{table:4_lps} shows four equivalent linear programs, arising from strong duality (see, e.g.,~\cite{bazaraa2011linear}) and switching between 'time' and 'frequency' domains. For the latter, we perform the change of variables, \(y \coloneqq Fx\) and \(t \coloneqq Fz\) respectively.

All formulations share the same structure: the optimization objective is determinisitic,
while the set of feasible solutions is random through the random circulant graph structure.
The following proposition introduces randomness to the objective, which is a crucial part of our argument.
\begin{lemma}
\label{lem:lp_with_g}
Let \(G\) be a dense RCG and \(\wt F\) be a subsampled DFT matrix, see~\Cref{def:f_tilde}. Let \(g \coloneqq Fb \in \R^n\), for \(b \in \{\pm 1\}^{n}\) with \(b_k = 1\) if \((0, k)\) is not an edge and \(-1\) otherwise. Then,
    \begin{equation}
\begin{aligned}
    \ltn(G) &= \max_{y \in \R^{n}} \braket{y, g}, \\ \text{ subject to }
    &\begin{cases} y_k = y_{n - k} \text{ for all } k \in [n] \setminus \{0\}, \\
    \norm{y}_1 = 1,
    y \geq \mathbf{0}, \\
    y \in \ker \wt F,
    \end{cases}
\end{aligned}
\end{equation}
\end{lemma}
\begin{proof}
We use the primal formulation in the frequency domain and observe that \(n y_0 = \braket{y, \sum_{k \in [n]} f_k}\).
Since feasible vectors \(y\) are orthogonal to \(\wt F\), i.e., \(y \in \ker \wt F\), after subtracting \(2 \sum_{(0, k) \in E(G)} \braket{y, f_k}\) from \(\braket{y, \sum_{k \in [n]} f_k}\) we obtain
\begin{equation}
\langle y, \sum_{k \in [n]} f_k \rangle = \langle y, \sum_{(0, k) \notin E(G)} f_k - \sum_{(0, k) \in E(G)} f_k \rangle = \langle y, g\rangle.
\end{equation}
\end{proof}

By the definition of graph \(G\), \(b_0 = 1\), and \(b_1, b_2, \ldots, b_{\frac{n - 1}{2}} \overset{\mathrm{iid}}{\sim} \mathrm{Unif}\{-1, 1\}\).
Since \(\max_{jk} \abs{F_{jk}} = 1\), we can bound \(\|g\|_{\infty}\), leading to the following upper bound on \(\ltn\).
\begin{lemma}
\label{lem:nlogn_ub}
Let \(G\) be a dense RCG. Then,
\begin{equation}
\Pro(\ltn(G) \leq 1 + 4\sqrt{n \log n}) \geq 1 - \frac{2}{n}.
\end{equation}
\end{lemma}
\begin{proof}
We show that each entry of \(g\) is small with high probability. Indeed, for any \(k \in [n]\),
\begin{equation}
\begin{aligned}
&\Pro(\abs{g_k} > 1 + 4 \sqrt{n \log n}) = \Pro(\abs{\langle f_k, b \rangle} > 1 + 4 \sqrt{n \log n}) \\
&\quad \leq \Pro\left(\abs*{\sum_{j = 1}^{(n - 1) / 2}X_j} > 2 \sqrt{n \log n}\right) \leq \frac{2}{n^2},
\end{aligned}
\end{equation}
where \(X_j \coloneqq \Re(F_{kj})b_j \in [-1, 1]\), and the last step follows from Hoeffding's inequality~(\Cref{lem:hoeffding}).
Applying union bound over \(k \in [n]\), we obtain
\begin{equation}
\label{eq:linf_norm}
\begin{aligned}
    &\Pro(\norm{g}_{\infty} > 1 + 4 \sqrt{n \log n}) \leq \frac{2}{n}.
\end{aligned}
\end{equation}
Thus, on a complement event, for any feasible vector \(y\) of~\Cref{eq:first_step}, we can simply upper bound \(\braket{y, g} \leq \norm{y}_{1} \norm{g}_{\infty} \leq 1 + 4\sqrt{n \log n}\), which finishes the proof.
\end{proof}

The upper bound in~\Cref{thm:main} would follow if we could show \(\max_k g_k = O(\sqrt{n \log \log n})\) with high probability.
However, this is too optimistic: since we expect that the coordinates of \(g\) behave like standard Gaussian random variables and are uncorrelated, 
we also expect that \(\max_k g_k = \Theta(\sqrt{n \log n})\). 
Fortunately, as the next lemma shows, only a vanishing fraction of entries is of order at least \(\sqrt{n \log \log n}\).
\begin{lemma}
\label{lem:few_large}
There exists a constant \(C > 0\), such that for \(\mathcal{I} \coloneqq \{k \in [n]: \abs{g_k} \geq C \sqrt{n\log \log n}\}\), it holds
    \begin{equation}
        \Pro\left(\lvert \mathcal{I} \rvert \leq \frac{n}{\log^{10} n}\right) \geq 1 - \frac{1}{\log^{10} n}.
    \end{equation}
\end{lemma}

\begin{proof}
    We express \(\abs{\mathcal{I}} = \sum_{k=0}^{n - 1} Y_k\), where \(Y_k = \mathbb{I}\{ \abs{g_k} \geq C \sqrt{n \log \log n}\}\). 
    Using Hoeffding's inequality we obtain, for \(C\) large enough,
    \begin{equation}
    \begin{aligned}
        \E \lvert \mathcal{I} \rvert &= \sum_{k = 0}^{n - 1} \Pro(\abs{g_k} \geq C \sqrt{n\log \log n}) \leq \frac{n}{\log^{20} n}, \\
    \end{aligned}
    \end{equation}
    where the constant on the right hand side is absorbed into logarithm, and its power is chosen for the technical reasons.
    Plugging this bound into Markov's inequality we get
    \begin{equation}
    \Pro\left(\lvert \mathcal{I} \rvert \geq \frac{n}{\log^{10} n}\right) \leq \frac{1}{\log^{10} n}.
    \end{equation}
\end{proof}

The constraint \(y \in \ker \wt F\) was so far only used 
to change the objective function from \(n y_0\) to \(\langle y, g \rangle\). 
Next lemma highlights another important consequence of this constraint,
namely, an upper bound on the \(\norm{y}_2\).
\begin{lemma}
\label{lem:rip}
    For large enough $n$, with probability at least \(1 - \frac{1}{n}\) all \(x \in \ker \wt F\) satisfy \(\norm{x}_2 \leq \frac{\log^2 n}{\sqrt{n}} \norm{x}_1\).
\end{lemma}
\begin{proof}
    We adapt the existing results in the literature regarding the RIP of the subsampled Fourier basis.
    
    Consider the following coupling: let \(\hat b \in \{0, 1\}^n\) with \(\hat b_0 = 0\) and \(\hat b_k \overset{\mathrm{iid}}{\sim} \mathrm{Ber}\left(\frac{\sqrt{2} - 1}{\sqrt{2}}\right)\) for \(k = 1, \ldots, n - 1\). Let \(\wt b \in \{0, 1\}^n\) be defined as follows:
    \begin{equation}
        \wt b_k = \begin{cases}
            0, \quad &\text{for }k = 0, \\
            1, \quad &\text{if }\hat b_k = 1 \text{ or } \hat b_{n - k} = 1, \text{ for } k \geq 1, \\
            0, \quad &\text{otherwise}.
        \end{cases}
    \end{equation}
    Note that (i) the distribution of \(\wt b\) is the same as the distribution of the adjacency vector for the vertex \(0\) in the random circulant graph \(G\) and (ii) \(\wt b_i = 0\) implies \(\hat b_i = 0\). 
    Let \(q \coloneqq \sum_k \hat b_k\). 
    Define \(\wh F \in \C^{q \times n}\) to be the matrix consisting of subsampled rows of \(F\) rescaled by \(1 / \sqrt{q}\),
    where the \(k\)-th row is included if and only if \(\hat b_k = 1\).

    To show that \(\wh F\) satisfies the RIP, we apply~\Cref{lem:rip_dft}. To ensure its requirements, we condition on the following two events. First, since we do not include row $0$ in our construction, we condition on the event that among the uniformly subsampled rows, row \(0\) is not present; this increases the probability of a bad event by at most a constant factor. Additionally, we condition on a high probability event that \(q \geq \lceil n / 4 \rceil\). \Cref{lem:rip_dft} then implies that there exist constants \(c > 0\) and \(0 < \varepsilon < 1/3\), such that with probability at least \(1 - 1/n\), 
    \(\wh F\) satisfies the RIP with parameters \(k = \frac{c n}{\log^3 n}\) and \(\varepsilon\). 
    
    On this event, by~\Cref{lem:rip_l2l1}, it follows that
    \begin{equation}
         \norm{x}_2 \leq \frac{C(\e)\log^{3/2} n}{\sqrt{cn}} \norm{x}_1 \leq \frac{\log^2 n}{\sqrt{n}} \norm{x}_1,
    \end{equation}
    for all \(x \in \ker \wh F\) and large enough $n$, where we absorbed the constants in the additional $(\log n)^{1/2}$ factor in the numerator. Since $\wh F$ consists of a subset of rows of $\wt F$, all \(x \in \ker \wt F\) are also in \(\ker \wh F\), so the proof is complete. 
\end{proof}

\begin{remark}[Alternative proof technique]
\Cref{lem:rip} also follows from an intermediate step in the proof of RIP of the subsampled Fourier matrix in~\cite{haviv2017restricted}. More specifically, in our notation \cite[Theorem~3.1]{haviv2017restricted} implies that $\|\wh F x\| \ge (1 - \varepsilon) \| F x\|^2_2 - C \varepsilon / k \| x \|_1^2$ with high probability, and since $x \in \ker \wh F$, it follows that $\|x\|_2 \le \frac{ \log^{2}n}{\sqrt{n}}$.
\end{remark}

Now we present the proof of our main result.
\begin{proof}[Proof of~\Cref{thm:main}]
We begin with the lower bound \(\E \ltn(G) \geq \sqrt{n}\). Since \(G\) is vertex-transitive, it holds that \(\ltn(G) \ltn(\overline{G}) = n\), see~\cite[Theorem 8]{lovasz1979shannon}. Therefore,
\begin{equation}
    \log n = \E \log \ltn(G)\ltn(\overline{G}) = 2 \E \log \ltn(G) \leq 2 \log \E \ltn(G),
\end{equation}
where we used the fact that \(G\) equals in distribution to \(\overline{G}\) together with Jensen's inequality and linearity of the expected value. Upon exponentiating we obtain
\begin{equation}
    \E \ltn(G) \geq \sqrt{n}.
\end{equation}

    To prove the upper bound, we use the LP formulation of the Lovász number as in~\Cref{lem:lp_with_g}. Let \(A\) denote the intersection of the events of~\Cref{lem:nlogn_ub,lem:rip}, with \(\Pro(A) \geq 1 - \frac{3}{n}\) from union bound, and let \(B\) denote the event of~\Cref{lem:few_large}. Since \(\E [\ltn \lvert \overline{A} \text{ or } \overline{B}] \Pro(\overline{A} \text{ or } \overline{B}) = O(1)\), we condition on \(A\) and \(B\) in the following. 
    For constant \(C\) defined in~\Cref{lem:few_large}, we split \(g\) into two parts, \(g_{\mathrm{small}}\) and \(g_{\mathrm{large}}\), where 
    \begin{equation}
    (g_{\mathrm{small}})_k = 
    \begin{cases} 
        g_k&\text{ if } \abs{g_k} < C \sqrt{n \log \log n}, \\
        0 &\text{ otherwise},
    \end{cases}
    \end{equation}
    and \(g_{\mathrm{large}} = g - g_{\mathrm{small}}\).
    Then, \(\braket{y, g} = \braket{y, g_{\mathrm{small}}} + \braket{y, g_{\mathrm{large}}}\). 
    We bound each term separately: first,
    \begin{equation}
        \braket{y, g_{\mathrm{small}}} \leq \norm{y}_1  \norm{g_{\mathrm{small}}}_{\infty} = O(\sqrt{n \log \log n}).
    \end{equation}
    On the event \(B\) we have that \(g_{\mathrm{large}}\) is \(\frac{n}{\log^{10} n}\)-sparse. 
    From~\Cref{eq:linf_norm} \(\norm{g_{\mathrm{large}}}_{\infty} = O(\sqrt{n \log n })\), which implies that \(\norm{g_{\mathrm{large}}}_2 = O(n / \log^4 n)\). Using Cauchy-Schwartz inequality together with~\Cref{lem:rip}, we bound the second term as follows:
    \begin{equation}
        \braket{y, g_{\mathrm{large}}} \leq \norm{y}_2 \norm{g_{\mathrm{large}}}_{2} \leq \frac{\log^2 n }{\sqrt{n}} \cdot \frac{n}{\log^4 n}= O(\sqrt{n}),
    \end{equation}
    which completes the proof.
\end{proof}

%% file: content/discussion.tex
\section{Discussion}
\label{sec:discussion}
Based on numerical observations, we formulate the following conjecture.

\begin{conjecture}
\label{conj:sharp}
    Let \(G\) be a dense random circulant graph. Then, 
    \begin{equation}
        \E \ltn(G) = (1 + o(1)) \sqrt{n}.
    \end{equation}
\end{conjecture}

Existing lower bounds against RIP (see~\cite{bandeira2018discrete,blasiok2019improved}) do not allow us to use our proof strategy for showing~\Cref{conj:sharp}. Indeed, there exist \(\frac{n}{\log n}\)-sparse vectors in the kernel of \(\wt F\), which contradicts the desired inequality \(\norm{y}_2 \leq \frac{C}{\sqrt{n}} \norm{y}_1\) for \(y \in \ker \wt F\). However, it is still possible that no \(c n\)-sparse \emph{entrywise positive} vector exists in the kernel of \(\wt F\), for small enough constant \(c > 0\). It is also plausible that constructing a feasible vector for the dual programs in~\Cref{table:4_lps} may lead to tighter upper bounds. We leave these questions for the future work.

\paragraph*{Paley graph}
A classical example of a circulant graph is \textit{Paley} graph. For a prime \(p \equiv 1 \mod 4\), it is defined as the graph on \(p\) vertices with vertices \(i\) and \(j\) connected if and only if \(i - j\) is a quadratic residue modulo \(p\), see~\cite{cohen1988clique, baker1996maximal}. Paley graphs are believed to exhibit certain \textit{pseudorandom} properties, and bounding its independence number is a long-standing open problem in number theory and combinatorics~\cite{hanson2021refined}. This quantity can be upper bounded by the Lovász number of a certain subgraph called 1-localization which is circulant \cite{kunisky2024spectral}.

Recently, several optimization based approaches were considered, see~\cite{kobzar2023revisiting,kunisky2024spectral,wanglower}. In~\cite{magsino2019linear}, a numerical evidence similar to~\Cref{conj:sharp} regarding subgraphs of Paley graph was observed, which if true, recovers the best known upper bound on the independence number due to~\cite{hanson2021refined}.

%% file: content/technical_lemmas.tex
\section{Useful inequalities}

\begin{lemma}[Hoeffding's inequality]
\label{lem:hoeffding}
    Let \(X_1, \ldots, X_n\) be independent random variables, such that \(\E X_i = 0\) and \(a \leq X_i \leq b\) almost surely. Then,
    \begin{equation}
        \Pro\left(\abs*{\sum_{i=1}^n X_i} \geq t\right) \leq 2\exp\left( - \frac{2 t^2}{n (b - a)^2}\right)
    \end{equation}
\end{lemma}

\begin{lemma}[RIP of subsampled DFT matrix,~\cite{haviv2017restricted}]
\label{lem:rip_dft}
Let \(F \in \C^{n \times n}\) be a DFT matrix: \(F_{jk} = \exp(-2 \pi i jk/n)\) for \(j, k \in [n]\). There exist \(c > 0\) and \(0 < \varepsilon < 1 / 3\), such that for all \(n\) large enough, a matrix consisting of \(q \geq \lceil n / 4 \rceil\) uniformly subsampled rows of \(F\) and rescaled by \(1 / \sqrt{q}\) has \((k, \varepsilon)\)-RIP for \(k = \frac{c n}{\log^3 n}\), with probability at least \(1 - 2^{\Omega(-\log^2 n)}\).
\end{lemma}
\begin{lemma}[e.g.~\cite{cahill2021robust}, Theorem 11]
\label{lem:rip_l2l1}
    If \(A \in \C^{m \times n}\) satisfies the RIP with parameters \(k\) and \(\varepsilon < 1/3\), then there exists \(C = C(\varepsilon)\), such that for any \(x \in \ker{A}\) we have that 
    \begin{equation}
        \norm{x}_2 \leq \frac{C}{\sqrt{k}}\norm{x}_1.
    \end{equation}
\end{lemma}

%% file: references.bib
@inproceedings{haviv2017restricted,
  title={The restricted isometry property of subsampled Fourier matrices},
  author={Haviv, Ishay and Regev, Oded},
  booktitle={Geometric Aspects of Functional Analysis: Israel Seminar (GAFA) 2014--2016},
  pages={163--179},
  year={2017},
  organization={Springer}
}

@article{aroranote,
  title={A note on the Lov{\'a}sz theta number of random graphs},
  author={Arora, Sanjeev and Bhaskara, Aditya}
}

@article{juhasz1982asymptotic,
  title={The asymptotic behaviour of Lov{\'a}sz'theta function for random graphs},
  author={Juh{\'a}sz, Ferenc},
  journal={Combinatorica},
  volume={2},
  number={2},
  pages={153--155},
  year={1982},
  publisher={Springer}
}

@article{lovasz1979shannon,
  title={On the Shannon capacity of a graph},
  author={Lov{\'a}sz, L{\'a}szl{\'o}},
  journal={IEEE Transactions on Information theory},
  volume={25},
  number={1},
  pages={1--7},
  year={1979},
  publisher={IEEE}
}

@article{baker1996maximal,
  title={Maximal cliques in the Paley graph of square order},
  author={Baker, RD and Ebert, GL and Hemmeter, J and Woldar, A},
  journal={Journal of statistical planning and inference},
  volume={56},
  number={1},
  pages={33--38},
  year={1996},
  publisher={Elsevier}
}

@article{cohen1988clique,
  title={Clique numbers of Paley graphs},
  author={Cohen, Stephen D},
  journal={Quaestiones Mathematicae},
  volume={11},
  number={2},
  pages={225--231},
  year={1988},
  publisher={Taylor \& Francis}
}

@article{kunisky2024spectral,
  title={Spectral pseudorandomness and the road to improved clique number bounds for Paley graphs},
  author={Kunisky, Dmitriy},
  journal={Experimental Mathematics},
  pages={1--28},
  year={2024},
  publisher={Taylor \& Francis}
}

@inproceedings{magsino2019linear,
  title={Linear programming bounds for cliques in Paley graphs},
  author={Magsino, Mark and Mixon, Dustin G and Parshall, Hans},
  booktitle={Wavelets and Sparsity XVIII},
  volume={11138},
  pages={440--447},
  year={2019},
  organization={SPIE}
}

@inproceedings{kobzar2023revisiting,
  title={Revisiting block-diagonal sdp relaxations for the clique number of the paley graphs},
  author={Kobzar, Vladimir A and Mody, Krishnan},
  booktitle={2023 International Conference on Sampling Theory and Applications (SampTA)},
  pages={1--5},
  year={2023},
  organization={IEEE}
}

@article{wanglower,
  title={Lower Bounds on Block-Diagonal SDP Relaxations for the Clique Number of the Paley Graphs and Their Localizations},
  author={Wang, Yifan and Shen, Yanling and Kobzar, Vladimir A}
}

@article{candes2006near,
  title={Near-optimal signal recovery from random projections: Universal encoding strategies?},
  author={Candes, Emmanuel J and Tao, Terence},
  journal={IEEE transactions on information theory},
  volume={52},
  number={12},
  pages={5406--5425},
  year={2006},
  publisher={IEEE}
}

@article{bandeira2018discrete,
  title={Discrete uncertainty principles and sparse signal processing},
  author={Bandeira, Afonso S and Lewis, Megan E and Mixon, Dustin G},
  journal={Journal of Fourier Analysis and Applications},
  volume={24},
  pages={935--956},
  year={2018},
  publisher={Springer}
}

@inproceedings{blasiok2019improved,
  title={An improved lower bound for sparse reconstruction from subsampled hadamard matrices},
  author={Blasiok, Jaroslaw and Lopatto, Patrick and Luh, Kyle and Marcinek, Jake and Rao, Shravas},
  booktitle={2019 IEEE 60th Annual Symposium on Foundations of Computer Science (FOCS)},
  pages={1564--1567},
  year={2019},
  organization={IEEE}
}

@article{demanet2014scaling,
  title={Scaling law for recovering the sparsest element in a subspace},
  author={Demanet, Laurent and Hand, Paul},
  journal={Information and Inference: A Journal of the IMA},
  volume={3},
  number={4},
  pages={295--309},
  year={2014},
  publisher={OUP}
}

@article{coja2005lovasz,
  title={The Lov{\'a}sz number of random graphs},
  author={Coja-Oghlan, Amin},
  journal={Combinatorics, Probability and Computing},
  volume={14},
  number={4},
  pages={439--465},
  year={2005},
  publisher={Cambridge University Press}
}

@article{hanson2021refined,
  title={Refined estimates concerning sumsets contained in the roots of unity},
  author={Hanson, Brandon and Petridis, Giorgis},
  journal={Proceedings of the London Mathematical Society},
  volume={122},
  number={3},
  pages={353--358},
  year={2021},
  publisher={Wiley Online Library}
}

@article{green2016counting,
  title={Counting sets with small sumset and applications},
  author={Green, Ben and Morris, Robert},
  journal={Combinatorica},
  volume={36},
  pages={129--159},
  year={2016},
  publisher={Springer}
}

@book{bazaraa2011linear,
  title={Linear programming and network flows},
  author={Bazaraa, Mokhtar S and Jarvis, John J and Sherali, Hanif D},
  year={2011},
  publisher={John Wiley \& Sons}
}

@article{green2005counting,
  title={Counting sets with small sumset, and the clique number of random Cayley graphs},
  author={Green*, Ben},
  journal={Combinatorica},
  volume={25},
  pages={307--326},
  year={2005},
  publisher={Springer}
}

@article{cahill2021robust,
  title={Robust width: A characterization of uniformly stable and robust compressed sensing},
  author={Cahill, Jameson and Mixon, Dustin G},
  journal={Excursions in Harmonic Analysis, Volume 6: In Honor of John Benedetto’s 80th Birthday},
  pages={343--371},
  year={2021},
  publisher={Springer}
}
